\documentclass[10pt]{amsart}
\usepackage{graphicx}
\usepackage{amsmath,amsthm}
\usepackage{amsfonts}
\usepackage{amscd}
\usepackage[all]{xy}
\usepackage{amssymb}
\usepackage[a4paper]{geometry}
\usepackage{subfigure}
\usepackage{mathtools}
\usepackage{upgreek}
\usepackage{enumitem}
\usepackage{newtxtext}
\usepackage{newtxmath}%Font
\usepackage{comment}

\usepackage[colorlinks=true,linkcolor=red,urlcolor=black]{hyperref}

\newtheorem{theorem}{Theorem}[section]

\newtheorem{lemma}{Lemma}[section]
\newtheorem{corollary}{Corollary}[section]
\newtheorem{definition}{Definition}[section]
\newtheorem{remark}{Remark}[section]

\def\cW{{\mathcal W}}
\def\cF{{\mathcal F}}

\title{On invariant holonomies between centers}
\author{Radu Saghin}
\date{\today}
\thanks{partially supported by ANID}

\address{Instituto de Matem\'atica, Pontificia Universidad Cat\'olica de Valpara\'iso,
Blanco Viel 596, Cerro Bar\'on, Valpara\'iso-Chile.}
\email{radu.saghin\@@pucv.cl}

\begin{document}

\begin{abstract}
We prove that for $C^{1+\theta}$, $\theta$-bunched, dynamically coherent partially hyperbolic diffeomorphisms, the stable and unstable holonomies between center leaves are $C^1$ and the derivative depends continuously on the points and on the map. Also for $C^{1+\theta}$, $\theta$-bunched partially hyperbolic diffeomorphism, the derivative cocycle restricted to the center bundle has invariant continuous holonomies which depend continuously on the map. This generalizes previous results by Pugh-Shub-Wilkinson, Burns-Wilkinson, Brown, Obata, Avila-Santamaria-Viana, Marin.
\end{abstract}

\maketitle

\tableofcontents

\setcounter{tocdepth}{1} \tableofcontents

\section{Introduccion}

Let $M$ be a compact smooth Riemannian manifold.

\begin{definition}\label{def:phd}
A diffeomorphism $f: M\to M$ of the compact Riemannian manifold $M$ is called \emph{partially hyperbolic} if the tangent bundle admits a continuous $Df$-invariant splitting $TM = E^s \oplus E^c\oplus E^u$ such that there exist continuous functions $0<\lambda_s(x)<\lambda_c^-(x)\leq\lambda_c^+(x)<\lambda_u(x)$, with $\lambda_s(x)<1<\lambda_u(x)$, satisfying the following conditions:
\begin{itemize}
\item[(1)] $\|Df(x)v^s\|\leq\lambda_s(x)$,
\item[(2)] $\lambda_c^-(x) \leq \|Df(x) v^c\| \leq \lambda_c^+(x)$,
\item[(3)] $\| Df(x) v^u\|\geq\lambda_u(x)$,
\end{itemize}
for every $x\in M$ and unit vectors $v^* \in E^*(x)$($* = s, c, u$).
\end{definition}

$E^s$ and $E^u$ are uniquely integrable, generating the stable and unstable foliations $\cW^s$ and $\cW^u$. A partially hyperbolic diffeomorphism is called \emph{dynamically coherent} if there exist invariant foliations $\cW^{cs}$ and $\cW^{cu}$ tangent to $E^{cs} = E^c \oplus E^s$ and $E^{cu} = E^c \oplus E^u$. In this case $\cW^{cs}$ is subfoliated by the stable and central foliations $\cW^s$ and $\cW^c$, while $\cW^{cu}$ is subfoliated by the unstable and center foliations $\cW^u$ and $\cW^c$.

\begin{definition}
A partially hyperbolic diffeomorphism is $\theta$-unstable bunched, $\theta>0$ if
\begin{equation}\label{global}
\lambda_u^{\theta}>\frac{\lambda_c^+}{\lambda_c^-}.
\end{equation}
Similarly we define $\theta$-stable bunching if $\lambda_s^{\theta}<\frac{\lambda_c^-}{\lambda_c^+}$, and $\theta$-bunched means both stable and unstable bunched.
\end{definition}

Given $f:M\rightarrow M$ partially hyperbolic and dynamically coherent, $p\in M$, $q\in\cW^u(x,f)$, we can define the unstable holonomy $h_{p,q,f}^u:\cW^c_{loc}(p)\rightarrow\cW^c(q)$ between the center leaves. We are addressing the question of differentiability of the holonomy along the center leaves, and the continuity of the derivative with respect to the points and the map.

\begin{theorem}\label{th:t1}
Suppose that $f$ is a $C^{1+\theta}$ partially hyperbolic diffeomorphism which is dynamically coherent and $\theta$-unstable bunched, $\theta\in(0,1]$. Then $h^u_{p,q,f}$ is $C^1$ and its derivative depends continuously on $f,p,q$ with $q\in\cW^u(p)$. A similar statement holds for the stable holonomy under the $\theta$-stable bunching condition.
\end{theorem}

\begin{remark}
The continuity means that if $f_n$ is inside a $C^{1+\theta}$ neighborhood of $f$ and converges to $f$ in the $C^{1}$ topology, $x_n$ converges to $x$, $y_n\in\cW^u_{loc}(x_n)$ and $y_n$ converges to $y$, then $Dh^u_{x_n,y_n,f_n}$ converges to $Dh^u_{x,y,f}$.
\end{remark}

Even if $f$ is not dynamical coherent, one can always construct fake foliations which are locally invariant under $f$ and are almost tangent to the invariant bundles (see \cite{BW10} for example). The fake foliations are a fundamental tool for  the study of ergodic properties of partially hyperbolic diffeomorphisms.

\begin{corollary}\label{cor}
Suppose that $f$ is a $C^{1+\theta}$ partially hyperbolic diffeomorphism which is $\theta$-unstable bunched, $\theta\in(0,1]$. Then the fake unstable holonomy between fake center leaves is uniformly $C^1$ (Lipschitz). A similar statement holds for the stable holonomy under the $\theta$-stable bunching condition.
\end{corollary}

Independently if $f$ is dynamically coherent or not, one can have invariant holonomies of the continuous cocycle defined by $Df|_{E^c}$.

\begin{definition}
Let $\mathcal E$ be a continuous vector bundle over $M$ and $F:\mathcal E\rightarrow \mathcal E$ a continuous linear cocycle over the partially hyperbolic diffeomorphism $f:M\rightarrow M$. An {\it invariant unstable holonomy for $F$} is a family of linear maps $\{H^u_{x,y}:\mathcal E(x)\rightarrow \mathcal E(y):\ \ x\in M, y\in\cW^u(x)\}$ satisfying the following conditions:
\begin{enumerate}
\item
$H^u_{x,x}=Id$, $H^u_{y,z}\circ H^u_{x,y}=H^u_{x,z}$;
\item
$F\circ H^u_{x,y}=H^u_{f(x),f(y)}\circ F$;
\item
$H^u_{x,y}$ is continuous in $x,y$ under the condition $y\in \cW^u_{loc}(x)$;
\end{enumerate}
The invariant stable holonomy is defined in a similar manner.
\end{definition}

One can also consider the projectivized bundle $\mathbb P\mathcal E$ over $M$, with fibers $\mathbb P\mathcal E(x)$ (the projective space of $\mathcal E(x)$), which is also a continuous bundle (with smooth fibers) over $M$. The projectivization of the cocycle $F$, $\mathbb PF$, is a continuous cocycle in $\mathbb P\mathcal E$. If $H$ is an invariant unstable holonomy for the cocycle $F$, then its projectivization $\mathbb PH$ is an invariant unstable holonomy for the cocycle $\mathbb PF$ (see for example \cite{ASV13} for more details on cocycles with holonomy and applications to the study of central Lyapunov exponents).

If $f$ is partially hyperbolic, then the center bundle forms a continuous (in fact H\"older if $f$ is $C^{1+\theta}$) vector bundle $\mathcal E^c(f)$ over $M$ and $Df|_{E^c(f)}$ is a continuous (H\"older) linear cocycle over $f$. A by-product of the proof of Theorem \ref{th:t1} is the following result.

\begin{theorem}\label{th:t2}
Suppose that $f$ is a $C^{1+\theta}$ partially hyperbolic diffeomorphism which is $\theta$-unstable bunched, $\theta\in(0,1]$. Then $Df|_{E^c}$ and $\mathbb PDf|_{\mathbb PE^c}$ have invariant unstable holonomies. The holonomies are also continuous with respect to the map in the $C^{1}$ topology restricted to a $C^{1+\theta}$ neighborhood of $f$. A similar statement holds for the stable holonomy under the $\theta$-stable bunching condition. If $f$ is dynamically coherent then the invariant holonomy coincides with the derivative of the holonomy between center leaves.
\end{theorem}

\begin{remark}
Theorem \ref{th:t1} and Theorem \ref{th:t2} work in particular for $C^2$ maps and the regular (1-) bunching condition.
\end{remark}

Let us make some historical remarks about these results. The differentiability of the holonomies along center leaves was established in \cite{PSW04} for $C^2$ partially hyperbolic diffeomorphisms which are 1-bunched, however the continuity of the derivative with respect to the points or the maps was not considered. The continuity of the derivative with respect to the points was proven in \cite{Ob18} under the additional assumptions of  $\alpha$-bunching and $\alpha$-pinching for some $\alpha>0$. The case of $C^{1+\theta}$ partially hyperbolic diffeomorphisms was addressed in several papers like \cite{BW05}, \cite{Br22}. The differentiability of the holonomy and the continuity of the derivative with respect to the point was obtained under the assumption of $\theta$-bunching together with more restrictive assumptions of pinching. The continuity of the derivative of the holonomy with respect to the map was not addressed to our knowledge.

Regarding the invariant holonomies, there are also various works establishing the existence and the continuity with respect to the map (the continuity with respect to the points is included in the definition) under the assumptions of $C^2$  smoothness, $\theta$-bunching and $\theta$-pinching (see for example \cite{ASV13},\cite{Ma16},\cite{LMY18},\cite{LMY19}). It seems to follow from the construction that in the dynamical coherent case the invariant holonomy of the center bundle cocycle coincides with the derivative of the regular holonomy between the centers of the original partially hyperbolic diffeomorphism.

Our contribution is to get rid of the unnecessary and restrictive pinching conditions, and to establish the full continuity (including with respect to the map) of the derivative of the holonomy and of the invariant holonomy, assuming only $\theta$-bunching and $C^{1+\theta}$ regularity of the map. We also give a unifyed presentation of both the differentiability of the holonomy between centers and the existence of invariant holonomies for the center derivative cocycle.

\subsection{Ideas of the proofs}

The main difficulty in the proof is the lack of sufficient regularity of the invariant bundles. The center bundle is H\"older continuous, but the H\"older exponent is smaller than $\theta$ in general, and this makes difficult to use the control which comes from the $\theta$-bunching and the $C^{\theta}$ regularity of the derivative. A first idea which we use is to consider the invariant holonomy together with a correction of the potential error coming from the variation of the center bundle with respect to the points (the projection from one bundle to the other, roughly along the unstable leaf is good enough). We can expect that the difference has better regularity along the unstable leaves. This observation together with a (more or less) standard application of the invariant section theorem \cite{HPS77} gives us the existence and continuity of the invariant holonomies (Theorem \ref{th:t2} without the identification with the derivative of the regular holonomy in the dynamically coherent case).

The differentiability of the regular holonomy requires more work. Previous works usually start with a good approximations of $\cW^u$ inside $\cW^{cu}$-leaves, and iterate it forward. Unfortunately again the leaves of $\cW^{cu}$ and $\cW^c$ are only $C^{1+\alpha}$ for some $\alpha<\theta$, and this fact limits the regularity of the approximation to $C^{1+\alpha}$, and consequently we loose the control when we iterate forward. The second idea of this paper is to start with a smooth approximation of both $\cW^u$ and $\cW^{cu}$-leaves and iterate it forward. It is important that these approximations are uniformly smooth, which makes the construction a bit more technical. When we iterate forward the approximation of $\cW^{cu}$-leaves and its subfoliation, the bunching condition helps us keep uniform $C^{1+\theta}$ control of the holonomy along the subfoliation. This argument will give us that the holonomy is Lipschitz, with uniform bounds on the Lipschitz constants.

In order to upgrade to differentiability we use the ideas from \cite{HPS77} on Lipschitz jets. The continuity of the derivative and the identification with the invariant holonomy is obtained again using the invariant section theorem.

\subsection{Several applications}

We list a couple of applications of the above results.

\begin{enumerate}
\item
The ergodicity of $C^{1+\theta}$ accessible $\theta$-center bunched partially hyperbolic diffeomorphisms can be obtained under weaker assumptions, without the condition that the invariant bundles are $C^{\theta}$ (\cite{BW10},\cite{RRU08}).
\item
The existence of invariant holonomies for the derivative cocycle on the center bundle for partially hyperbolic diffeomorphisms can be also obtained with weaker assumptions, without the $\theta$-pinching condition (and in $C^{1+\theta}$ regularity). This applies for example to various results concerning the continuity and the non-vanishing of central exponents of partially hyperbolic diffeomorphisms with two dimensional center (\cite{ASV13}, \cite{Ma16}, \cite{LMY18}, \cite{LMY19}, \cite{KS13})
\item
We establish the continuity of the derivative of the holonomies with respect to the points and the map, under more general conditions. This is a useful tool which can be applied in order to obtain  perturbation results related to the uniqueness of u-Gibbs or MMEs for some classes of partially hyperbolic diffeomorphisms (for example along the lines of \cite{Ob21}, \cite{COP22}, \cite{LSYY}) or related to the accessibility of partially hyperbolic diffeomorphisms (\cite{LP21}).
\end{enumerate}

\subsection{Organization of the paper}

In Section 2 we present some tools which we will use in the proof. In particular we discuss the regularity of the holonomy along a subfoliation of a submanifold, and how to approximate immersed submanifolds with smooth ones. In Section 3 we present the proofs.

\section{Preparations}

\subsection{Regularity of holonomy along a subfoliation: some general comments.}
We will start with a discussion about the regularity of the (derivative of) holonomy along a subfoliation of a submanifold in $\mathbb R^d$.

Assume that we have a $C^1$ embedded submanifold $\cW$ inside $\mathbb R^d$. Assume that $\mathcal F$ is a $C^1$ subfoliation of $\cW$. Given two points $x,y$ on the same leaf of $\mathcal F$, and two transversals $T_x$, $T_y$ to $\cF$ inside $\cW$ passing through $x$ and $y$, let $h^{\mathcal F}_{T_x,T_y}:T_x\rightarrow T_y$ be the holonomy given by $\mathcal F$.

Let $D_x=T_xT_x$ and $D_y=T_yT_y$ the tangent planes to $T_x$, $T_y$ in $x$ and $y$. Let $Dh^{\cF}_{T_x,T_y}:D_x\rightarrow D_y$ be the derivative of the holonomy $h^{\mathcal F}_{T_x,T_y}$. Clearly it depends only on $D_x$ and $D_y$ and not on the transversals $T_x$ and $T_y$, this is why we will also use the notation $Dh^{\cF}_{D_x,D_y}$. Given a decomposition $A\oplus B=\mathbb R^d$, we denote by $p_A^B:\mathbb R^d\rightarrow A$ the projection to $A$ parallel to $B$. If we want to specify that we consider the restriction of $p_A^B$ to a subspace $A'$ we will denote it $p_{A',A}^B$.

Let $d_{\mathcal F}$ be the distance induced on the leaves of $\mathcal F$.
\begin{definition}
Let $x\in\cW$, $\Delta$ be a continuous cone field inside $T\cW$ uniformly transverse to $\cF$, $E_x$ transverse to $\Delta_x$ and $\delta>0$. We say that {\it $Dh^{\mathcal F}$ is $(C_{\mathcal F},\theta)$-H\"older along $\mathcal F$ at $x$ with respect to $\Delta,E_x$ and at scale $\delta$} if
\begin{equation}\label{eq:reg}
\left\|Dh^{\mathcal F}_{D_x,D_y}(x)-p_{D_x,D_y}^{E_x}\right\|\leq C_{\mathcal F}d_{\mathcal F}(x,y)^{\theta},\ \ \forall y\in\mathcal F_{\delta}(x),\ \forall D_x\in\Delta_x,D_y\in\Delta_y.
\end{equation}

If instead of $\mathbb R^d$ we are in a smooth Riemannian manifold, the definition is similar, with the requirement that the condition \ref{eq:reg} holds in an exponential chart at $x$ of size $\delta$.
\end{definition}

Let us remark that given a $C^2$ submanifold $\cW$ with a $C^2$ subfoliation $\cF$, the continuous cone field $\Delta$, and a subspace $E_x$ containing $T_x\cF(x)$, there exist $C_{\cF},\delta>0$ such that $Dh^{\mathcal F}$ is $(C_{\mathcal F},\theta)$-H\"older along $\mathcal F$ at $x$ with respect to $\Delta,E_x$ and at scale $\delta$ (we can actually take $\theta=1$). The following lemma explains this fact in more detail.

We need a bound on the transversality between $E_x$ and $\Delta$ at the scale $\delta$:
$$
t(E_x,\Delta,\delta)= \sup \left\{\frac 1{\sin(\angle(E_x,D_y))}:\ y\in\cW_\delta(x), D_y\in\Delta_y\right\}.
$$
In particular we have
$$
\|p_{D_y}^{E_x}\|\leq t(E_x,\Delta,\delta) \hbox{ for all $y$ such that } d(x,y)<\delta.
$$
We aslo consider a bound on the transversality between $\Delta$ and $\cF$:
$$
t(\cF,\Delta)= \sup \left\{\frac 1{\sin(\angle(T_y\cF(y),D_y))}:\  D_y\in\Delta_y\right\}.
$$

We say that $\phi:\mathbb R^d\rightarrow\mathbb R^d$ is a {\it linear parametrization of ($\cW,\cF$)} if $\phi(\mathbb R^{\dim\cW}\times\{0\}^{d-\dim\cW})=\cW$ and $\phi(\mathbb R^{\dim\cF}\times\{b\}\times\{0\}^{d-\dim\cW})=\cF(\phi(0,b,0)), \forall b\in\mathbb R^{\dim\cW-\dim\cF}$ ($\phi$ basically straightens both $\cW$ and $\cF$). If $\phi$ is defined only between balls of radius $\delta$ at the origin and $x$ we say that it is a {\it $\delta$-linear parametrization of ($\cW,\cF$) at $x$}.

\begin{lemma}\label{le:cf}
Let $\cW$ be a $C^2$ submanifold in $\mathbb R^d$ and $\cF$ a $C^2$ subfoliation of $\cW$. Let $\Delta$ be a continuous cone field in $T\cW$ transferse to $\cF$, $x\in\cW$ and $E_x$ a subspace containing $T_x\cF(x)$ and transverse to $\Delta_y$ for all $y\in\cF_{\delta}(x)$ for some $\delta>0$.  Let $\phi$ be a $C^2$ $\delta$-linear parametrization of ($\cW,\cF$) at $x$. Then $Dh^{\mathcal F}$ is $(C_{\mathcal F},\theta)$-H\"older along $\mathcal F$ at $x$ with respect to $\Delta,E_x$ and at scale $\delta$ for $C_{\cF}=\|\phi\|_{C^{1+\theta}}^2\cdot\|\phi^{-1}\|_{C^1}^{2+\theta}\cdot t(E_x,\Delta,\delta)\cdot t(\cF,\Delta)\cdot\delta^{\theta}$.
\end{lemma}

\begin{proof}
Denote $\tilde*$ the pushed forward under $\phi^{-1}$ of the objects $*$. Observe that
$$Dh^{\tilde\cF}_{\tilde D_x,\tilde D_y}=p^{\tilde E_x}_{\tilde D_x,\tilde D_y}$$
because $\tilde E_x$ contains the plane parallel to the linear foliation $\tilde\cF$. Denote $D'=D\phi(\tilde x)\tilde D_y$ and $\tilde D'=D\phi^{-1}(x)D_y$. Since $D\phi\tilde E_x=E_x$ we have
\begin{eqnarray*}
p^{E_x}_{D_x,D_y}&=&D\phi(\tilde x)|_{\tilde D'}\circ p^{\tilde E_x}_{\tilde D_x,\tilde D'}\circ D\phi^{-1}(x)|_{D_x}\\
&=&D\phi(\tilde x)|_{\tilde D'}\circ p^{\tilde E_x}_{\tilde D_y,\tilde D'}\circ p^{\tilde E_x}_{\tilde D_x,\tilde D_y}\circ D\phi^{-1}(x)|_{D_x}\\
&=&p^{E_x}_{D',D_y}\circ D\phi(\tilde x)|_{\tilde D_y}\circ p^{\tilde E_x}_{\tilde D_x,\tilde D_y}\circ D\phi^{-1}(x)|_{D_x}.
\end{eqnarray*}
Then
\begin{eqnarray*}
\|Dh^{\cF}_{D_x,D_y}-p^{E_x}_{D_x,D_y}\|&=&\|D\phi(\tilde y)|_{\tilde D_y}\circ Dh^{\tilde\cF}_{\tilde D_x,\tilde D_y}\circ D\phi^{-1}(x)|_{D_x}-p^{E_x}_{D_x,D_y}\|\\
&=&\|\left(D\phi(\tilde y)|_{\tilde D_y}-p^{E_x}_{D',D_y}\circ D\phi(\tilde x)|_{\tilde D_y}\right)\circ p^{\tilde E_x}_{\tilde D_x,\tilde D_y}\circ D\phi^{-1}(x)|_{D_x}\|\\
&\leq&\|p^{E_x}_{D_y}\|\cdot\|D\phi(\tilde y)-D\phi(\tilde x)\|\cdot\|p^{\tilde E_x}_{\tilde D_y}\|\cdot\|D\phi^{-1}(x)\|\\
&\leq&\frac{t(E_x,\Delta,\delta)\cdot\|\phi\|_{C^{1+\theta}}\cdot\|\phi^{-1}\|_{C^1}^{1+\theta}\cdot  \delta^{\theta}}{\sin(\angle(\tilde D_y,T\tilde F))}\\
&\leq&\|\phi\|_{C^{1+\theta}}^2\cdot\|\phi^{-1}\|_{C^1}^{2+\theta}\cdot t(E_x,\Delta,\delta)\cdot t(\cF,\Delta)\cdot\delta^{\theta}.
\end{eqnarray*}
We used the fact that 
$$
\sin(\angle(D_y,T_yF))\leq \sin(\angle(\tilde D_y,T\tilde F))\|D\phi\|\cdot\|D\phi^{-1}\|.
$$
\end{proof}

We want to study the behavior of the regularity of foliations under the push-forward of a diffeomorphism. Assume that $\cW$ is contained in the open set $U$ and $f:U\rightarrow F(U)$ is a $C^{1+\theta}$ diffeomorphism. We will use the following notations for the bounds of $Df$ along $\Delta$ and $T\mathcal F$:
\begin{eqnarray*}
\lambda_{\Delta}^+(f,x,\delta)&:=&\sup_{d(x,y)<\delta}\|Df(y)|_{\Delta_y}\|;\\
\lambda_{\Delta}^-(f,x,\delta)&:=&\left(\sup_{d(x,y)<\delta}\|(Df(y)|_{\Delta_y})^{-1}\|\right)^{-1};\\
\lambda_{\mathcal F}(f,x,\delta)&:=&\left(\sup_{d(x,y)<\delta}\|(Df(y)|_{T_y\mathcal F(x)})^{-1}\|\right)^{-1};\\
\end{eqnarray*}

The following lemma is one of the main tools behind our proof. It keeps track on how the constant $C_{\cF}$ changes under iterations.

\begin{lemma}\label{le:main}
Let $\mathcal F$ be a foliation as above such that $Dh^{\mathcal F}$ is $(C_{\mathcal F},\theta)$-H\"older along $\mathcal F$ at $x\in\mathbb R^n$ with respect to $\Delta,E_x$ and at scale $\delta$. Let $f:U\rightarrow f(U)\subset\mathbb R^n$ be a $C^{1+\theta}$ diffeomorphism. Then for any $\Delta'\subset f_*\Delta$ and $\delta'<\lambda_{\cF}(f,x,\delta)\delta$, $Dh^{f_*\mathcal F}$ is $(C_{f_*\mathcal F},\theta)$-H\"older along $f_*\mathcal F$ at $f(x)$ with respect to $\Delta',f_*E_x$ and at scale $\delta'$, where
\begin{equation}
C_{f_*\mathcal F}=\frac{\lambda_\Delta^+(f,x,\delta)C_{\mathcal F}+t(E_x,\Delta,\delta)t(f_*E_x,\Delta',\delta')\|Df\|_{C^{\theta}}}{\lambda_\Delta^-(f,x,\delta)\lambda_{\mathcal F}(f,x,\delta)^{\theta}},
\end{equation}
\end{lemma}

\begin{proof} Denote $E'_x=Df(x)E_x, D'_x=Df(x)D_x, D'_y=Df(y)D_y, \tilde D=Df(x)D_y$.

Since $Df(x)$ takes the decomposition $E_x\oplus D_x$ to $E'_x\oplus D'_x$, we have that
$$
p_{D'_x,D'_y}^{E'_x}\circ Df(x)|_{D_x}=p_{\tilde D,D'_y}^{E'_x}\circ Df(x)|_{D_y}\circ p_{D_x,D_y}^{E_x}.
$$

We also have
$$
d_{f_*\mathcal F}(f(x),f(y))\geq\lambda_{\mathcal F}(f)d(x,y).
$$
For simplicity we will use the notation $\lambda_{\Delta}^{\pm}, \lambda_{\cF}$. We have
\begin{align*}
\left\|Dh^{f_*\mathcal F}_{D'_x,D'_y}-p_{D'_x,D'_y}^{E'_x}\right\|=&\left\|Df(y)|_{D_y}\circ Dh^{\mathcal F}_{D_x,D_y}(x)\circ \left(Df(x)|_{D_x}\right)^{-1}-p_{D'_x,D'_y}^{E'_x}\right\|\\
\leq&\left\|Df(y)|_{D_y}\circ\left(Dh^{\mathcal F}_{D_x,D_y}(x)-p_{D_x,D_y}^{E_x}\right)\circ \left(Df(x)|_{D_x}\right)^{-1}\right\|+\\
&+\left\|Df(y)|_{D_y}\circ p_{D_x,D_y}^{E_x}\circ \left(Df(x)|_{D_x}\right)^{-1}-p_{D'_x,D'_y}^{E'_x}\right\|\\
\leq&\left\|\left(Df(y)|_{D_y}\circ p_{D_x,D_y}^{E_x}-p_{D'_x,D'_y}^{E'_x}\circ Df(x)|_{D_x}\right)\left(Df(x)|_{D_x}\right)^{-1}\right\|+\\
&+\frac{\lambda_\Delta^+}{\lambda_\Delta^-}C_{\mathcal F}d_{\mathcal F}(x,y)^{\theta}\\
\leq&\frac{\lambda_\Delta^+}{\lambda_\Delta^-}C_{\mathcal F}d_{\mathcal F}(x,y)^{\theta}+\frac{1}{\lambda_\Delta^-}\left\|Df(y)|_{D_y}\circ p_{D_x,D_y}^{E_x}-p_{\tilde D,D'_y}^{E'_x}\circ Df(x)|_{D_y}\circ p_{D_x,D_y}^{E_x}\right\|\\
\leq&\frac{\lambda_\Delta^+}{\lambda_\Delta^-}C_{\mathcal F}d_{\mathcal F}(x,y)^{\theta}+\frac{1}{\lambda_\Delta^-}\|p_{D'_y}^{E'_x}\|\cdot \|Df(y)|_{D_y}-Df(x)|_{D_y}\| \cdot\|p_{D_x,D_y}^{E_x}\|\\
\leq&\left(\frac{\lambda_\Delta^+C_{\mathcal F}}{\lambda_\Delta^-\lambda_{\mathcal F}^{\theta}}+\frac{t(E_x,\Delta,\delta)t(f_*E_x,\Delta',\delta')\|Df\|_{C^{\theta}}}{\lambda_\Delta^-\lambda_{\mathcal F}^{\theta}}\right)d_{f_*\mathcal F}(f(x),f(y))^{\theta}.
\end{align*}
\end{proof}

\subsection{Smooth approximations of invariant submanifolds.}\label{ss:app}

The center-unstable leaves $\cW^{cu}$ of the partially hyperbolic diffeomorphism $f$ are subfoliated by the unstable leaves $\cW^u$, but unfortunately they are not smooth enough in order to carry out the ideas from the previous sub-section. This is why we need to construct smooth approximations of the center-unstable leaves, together with a smooth approximation of the unstable subfoliation. We need to approximate pieces of $\cW^{cu}$ which are arbitrarily large in the center direction, while making sure that the $C^2$ bounds of the approximations are uniform. The reader can keep in mind some specific examples where the smooth approximations are more or less straightforward: fake foliations -- $cu$-subspace subfoliated by $u$-subspaces; perturbations of linear maps -- the linear foliations of the original linear map. The case of partially hyperbolic diffeomorphisms which are fibered over hyperbolic homeomorphisms is also easier, because the center leaves are uniform $C^{1+\alpha}$ embeddings of the same compact fiber, and one can use standard smooth approximation. Our construction is a bit more technical because we want to include possible large pieces of the center manifolds with possible complicated topology.

Let us make some preparations.

\begin{definition}
A $C^r$ submanifold $\cW$ has size greater than $\delta$ at $x$ if within the exponential chart at $x$, $W$ contains the graph of a $C^r$ function $g$ from the ball of radius $\delta$ in $T_x\cW$ to the orthogonal complement $T_x\cW^{\perp}$.

If the ball of radius $\delta$ at $x$ in the $C^r$ submanifold $\cW$ can be written, in an exponential chart at $x\in M$, as the graph of a $C^r$ function $g$ from an open subset of $T_x\cW$ to the orthogonal complement $T_x\cW^{\perp}$, then the $(C^r,x,\delta)$ size of $\cW$ is $\|\cW\|_{C^r,x,\delta}=\|g\|_{C^r}$.
\end{definition}

Eventually modifying the Riemannian metric, we can assume that the invariant subspaces are close to orthogonal.
\begin{definition}
The continuous cone field $\Delta^*_{\epsilon}$ over $M$ is defined in the following way: $\Delta^*_{\epsilon}(x)$ contains the subspaces of $T_xM$ which have the same dimension as $E^*(x)$ and are $\epsilon$-close to $E^*(x)$, $*\in\{s,c,u,cs,su,cu\}$.
\end{definition}
For $\epsilon$ small we have that $\Delta^u_{\epsilon}$ and $\Delta^{cu}_{\epsilon}$ are forward invariant while $\Delta^s_{\epsilon}$ and $\Delta^{cs}_{\epsilon}$ are backward invariant.

Fix $\epsilon_0,\delta_0>0$ and a $C^{1+\theta}$ neighborhood $\mathcal U(f)$ of $f$ such that
\begin{itemize}
\item
the $s,sc$-cones of size $2\epsilon_0$ are backwards invariant while the $u,cu$-cones are forward invariant, for all $g\in\mathcal U(f)$;
\item
the cone fields $\Delta^*_{2\epsilon_0}$ are uniformly transverse for all $*\in\{s,c,u,cs,su,cu\}$ at the scale $\delta_0$, meaning that for every $x\in M$, within the exponential chart at $x$, $\Delta^*_{2\epsilon_0}(x)$ and $\Delta^{*'}_{2\epsilon_0}(y)$ are uniformly transverse for $d(x,y)<\delta_0$: $t(\Delta^*_{2\epsilon_0},\Delta^{*'}_{2\epsilon_0},\delta_0)<2$;
\item
The bunching condition holds at the ($2\epsilon_0,\delta_0$)- scale for all $g\in\mathcal U(f)$, meaning that
\begin{equation}\label{eq:bunch}
\frac{\lambda^+_{\Delta^c_{2\epsilon_0}}(g,x,\delta_0)}{\lambda^-_{\Delta^c_{2\epsilon_0}}(g,x,\delta_0)\lambda^-_{\Delta^u_{2\epsilon_0}}(g,x,\delta_0)^{\theta}}<\mu<1,\ \ \forall g\in\mathcal U(f),\forall x\in M.
\end{equation}
\item
The center bundle is uniformly $C^{\alpha}$ and the local center manifolds are uniformly $C^{1+\alpha}$, meaning that there exists $C_{\alpha}>0$ such that for every $g\in\mathcal U(f)$ and every $x,y\in M$, $d(E^c(x,g),E^c(y,g))\leq C_{\alpha}d(x,y)^{\alpha}$ and $\|\cW^c_{4\delta_0}(x,g)\|_{C^{1+\alpha},x,4\delta_0}<C_{\alpha}$.
\end{itemize}

The following lemma is an immediate consequence of the transversality. We say that the submanifold $W$ is tangent to the cone field $\Delta$ if $T_yW\in\Delta(y)$, for all $y\in W$.
\begin{lemma}[Local product structure]\label{le:ps}
There exist $\delta_p>0$ such that for any $0<\delta\leq\delta_p$, any $x,y\in M$ with $d(x,y)<\delta$, any $W_{2\delta}(x)$ $C^1$ manifold of size $2\delta$ at $x$ tangent to $\Delta^*_{\epsilon_0}$, and any $W_{2\delta}(y)$ $C^1$ manifold of size $2\delta$ at $y$ tangent to $\Delta^{*'}_{\epsilon_0}$, where $*$ and $*'$ are complementary combinations of $\{s,c,u\}$, then $W_{2\delta}(x)$ and $W_{2\delta}(y)$ intersect transversally in a unique point.
\end{lemma}

Now we are ready to construct the smooth approximations..

\subsubsection{Smooth uniform approximation of center manifolds}

The first step is to approximate large pieces of center manifolds with smooth ones, while keeping control on the smoothness of the approximations.

Fix a smooth approximation $\tilde E^{su}$ inside $\Delta^{su}_{\epsilon_0}$. There exists $0<\epsilon_1<\epsilon_0$ such that for every $p\in M$, the family $\{\exp_x(B_{\epsilon_1}\tilde E^{su}(x)):\ x\in \cW^c(p)\}$ subfoliate a tubular neighborhood of $\cW^c(p)$. Let $\tilde h^{su}$ be the holonomy generated by this subfoliation.

\begin{lemma}
For any $\epsilon>0$ small enough, any $p\in M$ and any $R>0$, there exists a smooth approximation of size $\epsilon$ of $\cW^c_R(p)$, meaning the following. There exists a smooth immersed manifold (possible with self-intersections) $\tilde\cW^c_{R,\epsilon}(p)$ tangent to $\Delta^c_{\epsilon}$, together with a local diffeomorphism $\tilde h^{su}_{\epsilon}$ given by the local $\tilde h^{su}$-holonomy.\\
Furthermore the approximations are uniform in the following sense. For every $x\in\tilde\cW^c_{R,\epsilon}(p)$ we have $\|\tilde\cW^c_{R,\epsilon}(p)\|_{C^{1+\alpha},x,\delta_0}\leq\tilde C_{\alpha}$ for some $\tilde C_{\alpha}$ independent of $p,R,\epsilon,f$, and $\|\tilde\cW^c_{R,\epsilon}(p)\|_{C^2,x,\delta_0}\leq\tilde C(\epsilon)$ for some $\tilde C(\epsilon)$ independent of $p,R,f$ (but depends on $\epsilon$).
\end{lemma}

\begin{proof}

Cover $M$ by a finite number of foliation charts of $\cW^c$ with center leaves of size $\delta_0$, say $U_1,U_2,\dots,U_k$. Then $W^c_R(p)$ is covered by finitely many leaves $\cW^c_{\delta_0}(x_i), 1\leq i\leq K$ from these foliation charts. Let $B_i=\cW^c_{3\delta_0}(x_i)$ and $\cW^c_R(p)\subset\cW_0=\cup_{i=1}^KB_i$.

Each $B_i$ is (contained in) the graph of a function $\gamma_i:B_{4\delta_0}E^c(x_i)\rightarrow E^{c^{\perp}}(x_i)$ with uniform $C^{1+\alpha}$ bounds. We will use the following standard regularization procedure.

Suppose that $\gamma:U\rightarrow E^{\perp}$ is $C^{1+\alpha}$, $B_{3\delta_0}E\subset U\subset E$, and it is also $C^{\infty}$ on some subset $V\subset U$. For any $\epsilon>0$ sufficiently small we can use the standard regularization and obtain $\gamma'$ which is $C^{\infty}$ and $C^{1}$ close to $\gamma$. Let $\rho$ be a smooth bump function which is one on $B_{2\delta_0}E$ and zero outside $B_{3\delta_0}E$. Use $\rho$ to interpolate between $\gamma'$ and $\gamma$ and obtain a new function $\tilde\gamma$ which is $C^{1+\alpha}$ on $B_{4\delta_0}E$, $C^{\infty}$ on $B_{2\delta_0}E\cup V$ and satisfies
\begin{itemize}
\item $\|\tilde\gamma-\gamma\|_{C^1}\leq\epsilon^{\alpha}\|\gamma\|_{C^{1+\alpha}}$ on $U$,
\item $\|\tilde\gamma\|_{C^{1+\alpha}}\leq 2\|\gamma\|_{C^{1+\alpha}}$ on $U$,
\item $\|\tilde\phi|_{B_{2\delta_0}E\cup V}\|_{C^2}\leq C(\epsilon)\max\{\|\phi\|_{C^{1+\alpha}},\|\phi|_V\|_{C^2}\}$,
\item $\tilde\phi=\phi$ outside $B_{3\delta_0}E$,
\end{itemize}
where $C(\epsilon)>2$ depends only on $\epsilon$ (and $\rho$).

We proceed with perturbing the leaves in $U_1$. Let $I_j=\{1\leq i\leq k: \cW^c_{\delta_0}(x_i)\in U_j\} $. By performing the perturbation described above to each $B_i, i\in I_1$ we obtain new submanifolds $B_i^1$ which are graphs of the functions $\gamma_i^1$ in exponential charts at $x_i$. The holonomy $\tilde h^{su}_1$ of length smaller that $2\epsilon^{\alpha} C_{\alpha}$ (in the exponential chart at $x_i$) is well defined between $B_i$ and $B_i^1$. Let
$$\cW_1=\cW_0\setminus(\cup_{i\in I_1}B_i)\cup(\cup_{i\in I_1}B_i^1).$$
The $\cW_1$ is an immersed submanifold of $M$, possible with self-intersections. Observe that we can extend $\tilde h^{su}_1$ as the identity outside $\cup_{i\in I_1}B_i^1$ and obtain a local diffeomorphism between $\cW_0$ and $\cW_1$. Let $A_1=\cup_{i\in I_1}Gr\left(\gamma_i^1|_{B_{2\delta_0}E^c(x_i)}\right)$ (the smooth part of $\cW_1$). For $i\notin I_1$ define $B_i^1=\tilde h^{su}_1(B_i)\subset \cW_1$. Then $\cW_1$ is the union of $B_i^1$, and each $B_i^1$ is related to $B_i$ by $\tilde h^{su}_1$. If $\pi_i$ is the projection on the first coordinate in the exponential chart at $x_i$, then $B_i^1$ is the graph of a function $\gamma_i^1:\pi_i(B_i^1)\rightarrow E^{c^{\perp}}(x_i)$ satisfying
\begin{itemize}
\item $B_{4\delta_0-\epsilon_1}E^c(x_i)\subset \pi_i(B_i^1)\subset B_{4\delta_0+\epsilon_1}E^c(x_i)$,
\item $\|\gamma_i-\gamma_i^1\|_{C^1}\leq\epsilon_1/2$,
\item $\|\gamma_i^1\|_{C^{1+\alpha}}\leq 2C_MC_{\alpha}$,
\item $\|\gamma_i^1|_{\pi_i(B_i^1\cap A_1)}\|_{C^2}\leq C(\epsilon)C_MC_{\alpha}$,
\end{itemize}
where $C_M$ depends on the Riemannian manifold $M$ and $\delta_0$ (measures the size of the change of coordinates between nearby exponential charts) and $\epsilon_1=2\epsilon^{\alpha}C_MC_{\alpha}$. Furthermore each $B_i$ is diffeomorphic to $B_i^1$ by the holonomy $\tilde h^{su}_1$ of length less that $\epsilon_1$. Observe that if $\epsilon$ is small enough so $\cW_1$ stays $C^1$ close to $\cW^c$, then we have that $A_1$ contains $\cup_{i\in I_1}\tilde h^{su}_1\left(\cW^c_{2\delta_0-2\epsilon_1}(x_i)\right)$.

Now we proceed with perturbing the leaves corresponding to $U_2$. In a similar manner we obtain a submanifold $\cW_2$ related to $\cW_1$ by the holonomy $\tilde h^{su}_2$ of length $\epsilon_2=4\epsilon^{\alpha}C_M^2C_{\alpha}$, and containing a smooth part $A_2\subset \cW_2$. $\cW_2$ is the union of $B_i^2=\tilde h^{su}_2(B_i^1)$, and each $B_i^2$ is the graph of a function $\gamma_i^2:\pi_i(B_i^2)\rightarrow E^{c^{\perp}}(x_i)$ satisfying
\begin{itemize}
\item $B_{4\delta_0-\epsilon_1-\epsilon_2}E^c(x_i)\subset \pi_i(B_i^2)\subset B_{4\delta_0+\epsilon_1+\epsilon_2}E^c(x_i)$,
\item $\|\gamma_i^1-\gamma_i^2\|_{C^1}\leq\epsilon_2/2$,
\item $\|\gamma_i^2\|_{C^{1+\alpha}}\leq 4C_M^2C_{\alpha}$,
\item $\|\gamma_i^2|_{\pi_i(B_i^1\cap A_2)}\|_{C^2}\leq C(\epsilon)^2C_M^2C_{\alpha}$.
\end{itemize}
If $\epsilon$ is small enough so that $\cW_2$ stays $C^1$ close to $\cW^c$, then
$$\cup_{i\in I_1\cup I_2}\tilde h^{su}_2\circ\tilde h^{su}_1\left(\cW^c_{2\delta_0-2\epsilon_1-2\epsilon_2}(x_i)\right)\subset A_2.$$

Continue by induction perturbing on each $U_i$ until we reach $U_k$. We get a submanifold $\cW_k$ related to $\cW_{k-1}$ by the holonomy $\tilde h^{su}_k$ of length $\epsilon_k=2^k\epsilon^{\alpha}C_M^kC_{\alpha}$, and containing a smooth part $A_k\subset \cW_k$. In particular $\cW_k$ is related to $\cW_0$ by the holonomy $\tilde h^{su}$ of length $\tilde\epsilon=\sum_{i=1}^k\epsilon_i$. Furthermore  $\cW_k$ is the union of $B_i^k=\tilde h^{su}_k(B_i^{k-1})$, and each $B_i^k$ is the graph of a function $\gamma_i^k:\pi_i(B_i^k)\rightarrow E^{c^{\perp}}(x_i)$ satisfying
\begin{itemize}
\item $B_{4\delta_0-\tilde\epsilon}E^c(x_i)\subset \pi_i(B_i^k)\subset B_{4\delta_0+\tilde\epsilon}E^c(x_i)$,
\item $\|\gamma_i^{k-1}-\gamma_i^k\|_{C^1}\leq\epsilon_k/2$, so $\|\gamma_i-\gamma_i^k\|_{C^1}\leq\tilde\epsilon/2$,
\item $\|\gamma_i^k\|_{C^{1+\alpha}}\leq 2^kC_M^kC_{\alpha}$,
\item $\|\gamma_i^k|_{\pi_i(B_i^1\cap A_k)}\|_{C^2}\leq C(\epsilon)^kC_M^kC_{\alpha}$.
\end{itemize}
If $\epsilon$ is small enough then we also have
$$\cup_{i\in I}\tilde h^{su}_k\circ\dots\circ\tilde h^{su}_1\left(\cW^c_{2\delta_0-2\tilde\epsilon}(x_i)\right)\subset A_k.$$

We can make this construction until the end for any $\epsilon$ small enough such that $\tilde\epsilon<\delta_0/2$ (and $\cW_k$ is close to $\cW^c$ so the estimates on the smooth part hold).

If $x=(a,\gamma_i^k(a))\in\cW^k$ (in a chart at $x_i$), let $x_0=(a, \gamma_i(a))\in\cW^c(x-i)$. We have $d(E^c(x),E^c(x_0))\leq C_{\alpha}\tilde\epsilon^{\alpha}$ and $d(E^c(x_0),Gr(D\gamma_i^k(a)))\leq \tilde\epsilon$, so $\cW_k$ is tangent to $\Delta^c_{\tilde\epsilon+\tilde\epsilon^{\alpha}C_{\alpha}}$.

Since $\lim_{\epsilon\rightarrow 0}\tilde\epsilon+\tilde\epsilon^{\alpha}C_{\alpha}=0$, the conclusions of the lemma hold with $\overline\epsilon=\tilde\epsilon+\tilde\epsilon^{\alpha}C_{\alpha}$, $\tilde\cW^c_{R,\overline\epsilon}(p)=\cW_k$, $\tilde h^{su}_{\overline\epsilon}=\tilde h^{su}_k\circ\dots\circ\tilde h^{su}_1$, $\tilde C_{\alpha}=2^kC_M^{k+1}C_{\alpha}$ and $\tilde C(\overline\epsilon)=C(\epsilon)^kC_M^{k+1}C_{\alpha}$.

\end{proof}

%\begin{remark}
%An alternative proof consists of covering $\cW^c_R(p)$ with finitely many (depending on $R$) balls of radius $\delta_r/2$ such that every point is contained in a finite number (independent of $R$) of balls of radius $2\delta_r$. Then one can perform the smooth approximation on each ball as in our proof.
%\end{remark}

\subsubsection{Smooth uniform approximation of center-unstable manifolds and of unstable foliation}
The second step is to use the smooth approximation of the center in order to construct smooth approximations of local center-unstable pieces together with a subfoliation close to the unstable one.

Fix a smooth global approximation $\tilde E^{u}$ of $E^u$, say within $\Delta^u_{\frac{\epsilon_0}10}$. We know from the previous step that $\tilde\cW^c_{R,\epsilon}(p)$ are uniformly $C^{1+\alpha}$ for all $p,R,\epsilon, f$. Then there exists $0<\delta_{\cF}<\min\{\delta_0,\delta_p\}$ such that, for every $p,R,\epsilon, f$, the family $\{\exp(B_{\delta_\cF}\tilde E^{u}(x)):\ x\in\tilde\cW^c_{R,\epsilon}(p)\}$ foliates a smooth submanifold inside a tubular neighborhood of $\tilde\cW^c_{R,\epsilon}(p)$; we denote this submanifold $\tilde\cW^{cu}_{R,\epsilon}(p)$, and the foliation $\cF^u_{R,\epsilon,p}$. By assuming that $\epsilon<\frac{\epsilon_0}{10}$ and eventually making $\delta_{\cF}$ smaller we have that $\tilde\cW^{cu}_{R,\epsilon}(p)$ is tangent to $\Delta^{cu}_{\epsilon_0}$ and $\cF^u_{R,\epsilon,p}$ is tangent to $\Delta^u_{\epsilon_0}$. We also have that $\tilde\cW^{cu}_{R,\epsilon}(p)$ and $\cF^u_{R,\epsilon,p}$ are uniformly $C^r$, $r\geq 2$ with respect to $p,R,f$ (the $C^r$ bounds do however depend on $\epsilon$).

\begin{lemma}\label{le:straight}
For any $\epsilon>0$ small enough there exists a constant $C_{\phi}(\epsilon)>0$ such that for every $p, R, f$ and any $x\in\tilde \cW^c_{R,\epsilon}(p)$, there exists a $\delta_{\cF}$-linear parametrization $\phi$ of ($\tilde\cW^c_{R,\epsilon}(p),\cF^u_{R,\epsilon,p}$) at $x$ with $\|\phi\|_{C^2},\|\phi^{-1}\|<C_{\phi}(\epsilon)$.
\end{lemma}

\begin{proof}
For simplicity of the notations we will work in an exponential chart at $x$, and we will make an abuse of notations using the same notation for the objects in $M$ and in the exponential chart.

Choose a decomposition $E_1\oplus E_2\oplus E_3=\mathbb R^d(=T_xM)$ with $E_1=T_x\tilde\cW^c_{R,\epsilon}(p)$, $E^2=\tilde E^u_x$ and $E_3$ orthogonal on $E_1,E_2$. Let $\alpha:B_{\delta_{\cF}}\mathbb R^{d}\times B_{\delta_{\cF}}E_2\rightarrow\mathbb R^d$ be a smooth parametrization of the family $\{\exp(B_{\delta_\cF}\tilde E^{u}(y)):\ y\in B_{\delta_{\cF}}\mathbb R^{d}\}$, in other words $\alpha(y,\cdot)$ is a parametrization of $exp(B_{\delta_\cF}\tilde E^{u}(y))$. We can assume that $\alpha(y,0)=y$, so $D_y\alpha(y,0)=Id_{\mathbb R^d}$, and $D_{y_2}\alpha(y,0)$ is a linear map from $E_2$ to $\tilde E^u_y$, uniformly bounded from zero and infinity. In particular $D_{y_2}\alpha(0,0)$ is an automorphism of $E_2$. The map $\alpha$ is $C^{\infty}$ and its size depends only on the Riemannian structure on $M$ and the choice of $\tilde E^u$.

Let $\gamma:B_{\delta_{\cF}}E_1\rightarrow E_2\oplus E_3$ be a smooth function such that its graph is the local manifold $\tilde\cW^c_{R,\epsilon}(p)$ in a neighborhood of $x$. Then $\gamma$ has the $C^{1+\alpha}$ size bounded by $2\tilde C_{\alpha}$ and its $C^2$ size bounded by $2\tilde C(\epsilon)$.

Let $\phi:B_{\delta_{\cF}}\mathbb R^d\rightarrow\mathbb R^d$,
$$
\phi(y_1,y_2,y_3)=(\alpha((y_1,\gamma(y_1)),y_2)+y_3.
$$
It is clear from the definition that $\phi$ is a $\delta_{\cF}$-linear parametrization $\phi$ of ($\tilde\cW^c_{R,\epsilon}(p),\cF^u_{R,\epsilon,p}$) at $x$. The $C^2$ size of $\phi$ is bounded by some $C_{\phi}(\epsilon)$ which depends on $\tilde C(\epsilon)$, the Riemannian structure of $M$ and the choice of $\tilde E^u$. The $C^{1+\alpha}$ size of $\phi$ is bounded by some constant which depends on $\tilde C(\alpha)$, the Riemannian structure of $M$ and the choice of $\tilde E^u$. 

We have
$$
D\phi(0)=\begin{bmatrix}
Id_{E_1} & 0 & 0\\
0 & D_{y_2}\alpha(0,0) & 0\\
0 & 0 & Id_{E_3}
\end{bmatrix}.
$$
The determinant is uniformly bounded away from zero, so eventually readjusting $\delta_F$ we have that the $C^1$ size of $\phi^{-1}$ is uniformly bounded. This finishes the proof.
\end{proof}

We claim that $\cW^u_{\delta}\cW^c_{R-r}(p)$ and $\tilde\cW^{cu}_{R,\epsilon}(p)$ are related by local stable holonomy for some $r>0$ and $\epsilon,\delta$ sufficiently small.

\begin{lemma}\label{le:hs}
Suppose that $\epsilon$ is small enough and $r-2\epsilon>\delta_{\cF}$, $\delta+\epsilon<\frac{\delta_{\cF}}4$. Then for any $p,R,f$, with $R>r$, the stable holonomy of size $2(\delta+\epsilon)$ gives a local homeomorphism from $\cW^u_{\delta}\cW^c_{R-r}(p)$ to (a subset of) $\tilde\cW^{cu}_{R,\epsilon}(p)$.
\end{lemma}

\begin{proof}
Let $x\in\cW^c_{R-r}(p)$ and $y\in \cW^u_{\delta}(x)$. Let $x'=\tilde h^{su}_{\epsilon}(x)\in\tilde\cW^c_{R,\epsilon}(p)$, so $d(x,x')<\epsilon$. Then $\tilde\cW^c_{R,\epsilon}(p)$ has size at $x'$ at least $r-2\epsilon>\delta_{\cF}$. This implies that $\tilde\cW^{cu}_{R,\epsilon}(p)$ has size at $x'$ at least $\frac{\delta_{\cF}}2$. On the other hand $d(y,x')\leq\delta+\epsilon<\frac{\delta_{\cF}}4$. The local product structure from Lemma \ref{le:ps} implies that $\cW^s_{2(\delta+\epsilon)}(y)$ intersects transversely the disk centered at $x'$ of size $2(\delta+\epsilon)<\frac{\delta_{\cF}}2$ in $\tilde\cW^{cu}_{R,\epsilon}(p)$ in a point $h^s_{2(\delta+\epsilon)}(y)$. Then $h^s_{2(\delta+\epsilon)}(y)$ is a local homeomorphism from $\cW^u_{\delta}\cW^c_{R-r}(p)$ to $\tilde\cW^{cu}_{R,\epsilon}(p)$.
\end{proof}

\section{Proofs}

We divide the proof of Theorem \ref{th:t1} in several steps.

\subsection{Approximation of the unstable holonomies: construction of $h^n_{p,q}$}

We start with the construction of an approximation of the unstable holonomy inside center-unstable leaves. From now on we fix $r=2\delta_{\cF}$ and $0<\epsilon=\delta<\frac{\delta_{\cF}}{10}$ small enough so that all the conclusions from sub-section \ref{ss:app} hold.

Let $p\in M$, $q\in\cW^u_{\delta}(p)$, $x\in\cW^c_{\delta}(p)$, $z=h^u_{p,q}(x)\in\cW^c_{2\delta}(q)$ and $R_n=3{\delta}\sup_{x\in M}\|Df^{-n}|_{E^c}\|+r$. We start iterating back by $f^{-n}$. Denote 
$$\cW_n=\cW^u_{\delta}\cW^{c}_{R_n-r}(f^{-n}(p)).$$
Observe that $f^{-n}(\cW^c_{2\delta}(p),f^{-n}(\cW^c_{2\delta}(p)\subset\cW_n$.

As in the previous section we consider the approximation $\tilde\cW^{cu}_{R_n,\epsilon}(f^{-n}(p)):=\tilde\cW_n$ and its subfoliation $\cF_{R_n,\epsilon,f^{-n}(p)}:=\cF_n$. Lemma \ref{le:hs} implies that the stable holonomy of size (smaller than) $4\delta$, $h^s_{4\delta}:\cW_n\rightarrow\tilde\cW_n$, is a local homeomorphism. Denote $\tilde*_n=h^s_{4\delta}(f^{-n}(*))\in\cW^n$ for $*\in\{p,q,x,z\}$.

Let $T_{\tilde p_n}=h^s_{4\delta}(f^{-n}(\cW^c_{2\delta}(p)))$ and $T_{\tilde q_n}=h^s_{4\delta}(f^{-n}(\cW^c_{2\delta}(q)))$, they are $C^{1+\alpha}$ transversals to the foliation $\cF_n$ in $\tilde\cW_n$ (they are in fact tangent to $\Delta^c_{\epsilon}$).

Now we iterate $\tilde\cW_n$ and $\cF_n$ forward by $f^n$. Denote $*_n=f^n(\tilde *_n)=h^s_{4\lambda_s^n\delta}(*)$,  for $*\in\{p,q,x,z\}$ (the stable holonomy commutes with $f$ and is uniformly contracted). Also denote $T_{p_n}=f^n(T_{\tilde p_n})=h^s_{4\lambda_s^n\delta}(\cW^c_{2\delta}(p))$ and $T_{q_n}=f^n(T_{\tilde q_n})=h^s_{4\lambda_s^n\delta}(\cW^c_{2\delta}(q))$, they are again $C^{1+\alpha}$ transversals to $f^n_*\cF_n$ inside $f^n\tilde\cW_n$.

The partial hyperbolicity implies that
\begin{itemize}
\item
$p_n,q_n$ and $x_n$ converge exponentially to $p,q,x$;
\item
$T_{p_n}$ and $T_{q_n}$ converge to $\cW^c_{2\delta}(p)$ and $\cW^c_{2\delta}(q)$ in the $C^1$ topology;
\item
$h^s_{4\lambda_s^n\delta}(\cW^{cu}_{loc}(p))\subset f^n\tilde\cW_n$ converges to $\cW^{cu}_{loc}(p)$ in the $C^1$ topology;
\item
$f^n_*\cF_{n,loc}$ converges to $\cW^u_{loc}$ in the following sense: if $a_n$ converges to $a$ then $f^n_*\cF_{n,loc}(a_n)$ converges to $\cW^u_{loc}(a)$ in the $C^1$ topology.
\end{itemize}

Let $T_{p_n}'=h^s_{4\lambda_s^n\delta}(\cW^c_{\delta}(p))\subset T_{p_n}$. Then for $n$ sufficiently large there exists a well defined holonomy of the foliation $f^n_*\cF_n$ between the transversals $T_{p_n}'$ and $T_{q_n}$. We denote this holonomy $h_{p_n,q_n}^{f^n_*\cF_n}$ and observe that it is $C^{1+\theta}$.

Define $h^n_{p,q}:\cW^c_{\delta}(p)\rightarrow\cW^c_{2\delta}(q)$,
\begin{equation}
h^n_{p,q}=h^{s^{-1}}_{4\lambda_s^n\delta}\circ(h^{f^n_*\cF_n}_{p_n, q_n})\circ h^s_{4\lambda_s^n\delta},
\end{equation}
In other words, in order to obtain $h^n_{p,q}(x)$ for $x\in\cW^c_{\delta}(p)$, we move with the stable holonomy of size $4\lambda_s^n\delta$ to $T_{p_n}'\subset f^n\tilde\cW_n$, then we move with the holonomy given by the foliation $f_*^n\cF_n$ of $f^n(\tilde\cW_n)$ between the $C^1$ transversals $T_{p_n}'$ and  $T_{q_n}$, and then we move back by the stable holonomy of size $4\lambda_s^n\delta$ to $\cW^c_{2\delta}(q)$.

Clearly $h^n_{p,q}$ is continuous, since the stable holonomies are H\"older continuous, while the holonomy $h^{f^n_*\cF_n}_{p_n,q_n}$ is $C^{1+\theta}$.

\subsection{$h^n_{p,q}$ converges uniformly to $h^u_{p,q}$.}

This follows immediately from the remarks in the previous section.

\subsection{$h^u_{p,q}$ is Lipschitz.}

We first show that $Dh^{f^n_*\cF_n}_{p_n,q_n}(x_n)$ is bounded uniformly in $n,x_n$.

Let $\tilde\Delta_n=\Delta^{cs}_{\epsilon_0}\cap T\tilde\cW_n\subset\Delta^c_{2\epsilon_0}$ be a cone field tangent to $\cW_n$. Let $E_{\tilde x_n}=E^s_{\tilde x_n}\oplus T_{\tilde x_n}\cF_n$. Since the cone fields $\Delta^*_{2\epsilon_0}$  are uniformly transverse at the scale $\delta_0$, we have $t(E_{\tilde x_n},\tilde\Delta_n,\delta)<2$ and $t(\cF_n,\tilde\Delta_n)<2$. In view of the Lemma \ref{le:straight} and Lemma \ref{le:cf} we have that $Dh^{\cF_n}$ is ($C_{\cF},\theta$)-H\"older along $\cF_n$ at $\tilde x_n$ with respect to $\tilde\Delta_n$, $E_{\tilde x_n}$ and at scale $\delta$, for some constant $C_{\cF}$ independent on $p,n,\tilde x_n$ and $g$ in $\mathcal U(f)$.

Let $\tilde\Delta_n^k=\Delta^{cs}_{\epsilon_0}\cup Tf^k\tilde\cW_n\subset\Delta^c_{2\epsilon_0}$ be a cone field tangent to $f^k\cW_n$. Observe that $\tilde\Delta_n^{k+1}\subset f_*\tilde\Delta_n^k$ because of the backward invariance of $\Delta^{cs}_{\epsilon_0}$. Since $Tf_*^k\cF_n$ stays tangent to $\Delta^u_{\epsilon_0}$, we also have uniform transversality between $\tilde\Delta_n^k$ and both $f_*^k\cF_n$ and $f_*^kE_{\tilde x_n}$ at scale $\delta$:  $t(f_*^kE_{\tilde x_n},\tilde\Delta_n^k,\delta)<2$ and $t(f_*^k\cF_n,\tilde\Delta_n^k)<2$. Due to the fact that $\cF_n$ is uniformly expanding, we can apply successively Lemma \ref{le:main} and using the bunching condition we conclude that $Dh^{f_*^n\cF_n}$ is ($C_0,\theta$)-H\"older along $f_*^n\cF_n$ at $f^n(\tilde x_n)=x_n$ with respect to $\tilde\Delta_n^n$, $f_*^nE_{\tilde x_n}$ and at scale $\delta$, for the constant $C_0=C_{\cF}+\frac{4\|Df\|_{C^{\theta}}}{(1-\mu)\|Df\|}$ independent of $p,n,\tilde x_n$. The constant also works for $g$ within the neighborhood $\mathcal U(f)$ of $f$ (eventually readjusting $\mathcal U(f)$ or $C_0$). Then  $Dh^{f^n_*\cF_n}_{p_n,q_n}(x_n)$ is bounded uniformly by some constant $L_0$, so $h^{f^n_*\cF_n}_{p_n,q_n}$ is Lipschitz with constant $L_0$ uniformly in $n$.

Now the fact that $h^u_{p,g}$ is Lipschitz is just a simple consequence of the fact that $h^{f^n_*\cF_n}_{p_n,q_n}$ are uniformly Lipschitz. We have that $d(h^{f^n_*\cF_n}_{p_n,q_n}(x_n),h^{f^n_*\cF_n}_{p_n,q_n}(x_n'))\leq L_0d(x_n,x_n')$ uniformly in $n$. Since $x_n$ converges to $x$, $x_n'$ converges to $x'$, $h^{f^n_*\cF_n}_{p_n,q_n}(x_n)$ converges to $h^u_{p,q}(x)$ and $h^{f^n_*\cF_n}_{p_n,q_n}(x_n')$ converges to $h^u_{p,q}(x')$, it follows that $d(h^u_{p,q}(x),h^u_{p,q}(x'))\leq L_0d(x,x')$.

\subsection{Estimate on the Lipschitz jet of $h^u_{p,q}$.}

Let us remind the definition of Lipschitz jets. Let $M,N$ be two metric spaces, $p\in M$, $q\in N$. Two functions $f,g:M\rightarrow N$ such that $f(p)=g(p)=q$ are equivalent if $\limsup_{x\rightarrow p}\frac{d(f(x),g(x))}{d(x,p)}=0$. The equivalence classes form the space $J(M,p,N,q)$ of Lipschitz jets at $p,q$. The distance between two Lipschitz jets is $d(J(f),J(g))=\limsup_{x\rightarrow p}\frac{d(f(x),g(x))}{d(x,p)}$, it can be infinite and is independent of the representatives $f$ and $g$. A Lipschitz jet is bounded if the distance to the jet of the constant function is finite. The space of bounded Lipschitz jets at $p,q$, $J^b(M,p,N,q)$, is a complete metric space. If $M,N$ are differentiable manifolds, then the space of differentiable Lipschitz jets at $p,q$, $J^d(M,p,N,q)$, is formed by the jets which have a representative which is differentiable. $J^d(M,p,N,q)$ is a closed subspace of $J^b(M,p,N,q)$.

For simplicity let us denote $D_{x_n}=T_{x_n}T_{p_n}$, $D_{y_n}=T_{y_n}T_{q_n}$,  where $y_n=h^{f^n_*\cF_n}_{p_n,q_n}(x_n)$, $E_{x_n}=f_*^nE_{\tilde x_n}$, $\overline h^n=h^{f^n_*\cF_n}_{p_n,q_n}$. We have
\begin{equation}\label{eq:lip}
\|D\overline h^n(x_n)-p^{E_{x_n}}_{D_{x_n},D_{y_n}}\|\leq C_0d(x_n,y_n)^{\theta}
\end{equation}
In particular we have that $D_{x_n}$ and $D_{y_n}$ converge exponentially to $E^c(x),E^c(z)$ when $n$ goes to infinity, while $E_{x_n}$ converges exponentially to $E^s(x)\oplus E^u(x)$.

We will work in an exponential chart at $p_n$, and we will make an abuse of notation keeping the notation of the points. Let $B_{p_n}, B_{q_n}$ be the balls or radius $\delta$ in $D_{p_n}, D_{q_n}$. We can choose $C^{1+\alpha}$ maps $\sigma_{p_n}:B_{p_n}\rightarrow T_{p_n}$ and $\sigma_{q_n}:B_{q_n}\rightarrow T_{q_n}$ such that
\begin{itemize}
\item
$p_n+x'-\sigma_{p_n}(x')\in E_{p_n},\ \forall x'\in B_{p_n}$;
\item
$q_n+y'-\sigma_{q_n}(y')\in E_{p_n},\ \forall y'\in B_{q_n}$.
\end{itemize}
In other words they are parametrisations of $T_{p_n}, T_{q_n}$ given by the projection from $B_{p_n}, B_{q_n}$ parallel to $E_{p_n}$. Using them we can define $g_n:T_{p_n}'\rightarrow T_{q_n}$, $g_n=\sigma_{q_n}\circ p^{E_{p_n}}_{D_{p_n},D_{q_n}}\circ\sigma_{p_n}^{-1}$. This means that $g_n$ has $Dg_n(p_n)=p^{E_{p_n}}_{D_{p_n},D_{q_n}}$. We will analyze the Lipschitz jets of $\overline h_n$ and $g_n$ at $p_n$.

We will use the notations $x_n=\sigma_{p_n}(x_n'), y_n=\sigma_{q_n}(y_n')$. We can see that
\begin{itemize}
\item
$\sigma_{p_n}(0)=p_n,\ \sigma_{q_n}(0)=q_n$;
\item
$D\sigma_{p_n}(0)=Id_{D_{p_n}},\ D\sigma_{q_n}(0)=Id_{D_{q_n}}$;
\item
$D\sigma_{p_n}(x_n')=p^{E_{p_n}}_{D_{p_n},D_{x_n}},\ D\sigma_{q_n}(y_n')=p^{E_{p_n}}_{D_{q_n},D_{y_n}}$;
\end{itemize}

Let $G_n=\sigma_{q_n}^{-1}\circ\overline h^n\circ\sigma_{p_n}-p^{E_{p_n}}_{D_{p_n},D_{q_n}}:B_{p_n}\rightarrow D_{q_n}$. We have that $G_n(0)=0$ and $\|DG_n(0)\|\leq C_0d(p_n,q_n)^{\theta}$. We have
\begin{align*}
\|DG_n(x_n')\|&=\|D\sigma_{q_n}^{-1}\circ D\overline h^n(x_n)\circ D\sigma_{p_n}(x_n')-p^{E_{p_n}}_{D_{p_n},D_{q_n}}\|\\
&=\|p^{E_{p_n}}_{D_{q_n}}\circ(D\overline h^n(x_n)-p^{E_{x_n}}_{D_{x_n},D_{y_n}})\circ p^{E_{p_n}}_{D_{x_n}}+p^{E_{p_n}}_{D_{q_n}}\circ(p^{E_{x_n}}_{D_{x_n},D_{y_n}}-p^{E_{p_n}}_{D_{x_n},D_{y_n}})\circ p^{E_{p_n}}_{D_{x_n}}\|\\
&\leq\|p^{E_{p_n}}_{D_{q_n}}\|\cdot \|p^{E_{p_n}}_{D_{x_n}}\|\cdot(C_0d(x_n,y_n)^{\theta}+2d(E_{p_n},E_{x_n}))\\
&\leq4(C_0d(x_n,y_n)^{\theta}+2d(E_{p_n},E_{x_n})).
\end{align*}

There exists $\gamma>0$ depending on $d(p,q)$ such that for all $n$ sufficiently large and all $x_n\in T_{p_n}'$ with $d(x_n,p_n)<\gamma$, we have
\begin{itemize}
\item
$d(x_n,y_n)<2d(p,q)$;
\item
$8d(E_{p_n},E_{x_n}))<C_0d(p,q)^{\theta}$.
\end{itemize}

We deduce that if $d(x_n,p_n)<\gamma$ then $\|DG(x_n')\|<5C_0d(p,q)^{\theta}$, or $G$ is Lipschitz with constant $5C_0d(p,q)^{\theta}$. Then
$$
d(\sigma_{q_n}^{-1}\circ\overline h^n\circ\sigma_{p_n}(x_n'),p^{E_{p_n}}_{D_{p_n},D_{q_n}}(x_n'))=d(G(x_n'),G(0))\leq 5C_0d(p,q)^{\theta}d(x_n',0)\leq 10C_0d(p,q)^{\theta}d(x_n,p_n).
$$
and furthermore
\begin{eqnarray*}
\sup_{d(x_n,p_n)<\gamma}\frac{d(\overline h^n(x_n),g_n(x_n))}{d(x_n,p_n)}&\leq& Lip(\sigma_{q_n})\sup_{d(x_n,p_n)<\gamma}\frac{d(\sigma_{q_n}^{-1}\circ \overline h^n(x_n),\sigma_{q_n}^{-1}\circ g_n(x_n))}{d(x_n,p_n)}\\
&\leq&20C_0d(p,q)^{\theta}
\end{eqnarray*}

In other words $d(J(\overline h^n),J(g_n))\leq20C_0d(p,q)^{\theta}$ in $J^b(T_{p_n},p_n,T_{q_n},q_n)$ (in fact in $J^d(T_{p_n},p_n,T_{q_n},q_n)$). Since $\gamma$ is independent of $n$ this relation can be passed to the limit when $n$ goes to infinity and we get
$$
\sup_{d(x,p)<\gamma}\frac{d(h^u_{p,q}(x),g_{p,q}(x))}{d(x,p)}\leq 20C_0d(p,q)^{\theta},
$$
where $g_{p,q}=\sigma_q\circ p^{E^{su}(p)}_{E^c(p),E^c(q)}\circ\sigma_p^{-1}$. This means that $d(J(h^u_{p,q}),J(g_{p,q}))\leq 20C_0d(p,q)^{\theta}$, for all $p\in M$ and $q\in\cW^u_{\delta}(x)$.

\begin{remark}
$g_{p,q}$ is differentiable and the derivative is $p^{E^{su}(p)}_{E^c(p),E^c(q)}$. The bound obtained also works for the neighborhood $\mathcal U(f)$.
\end{remark}

\subsection{$h^u_{p,q}$ is differentiable.}
We will use the invariant section theorem. Let $q\in\cW^u_{\delta}(p)$, $q\neq p$. For simplicity let us denote $g_{f^n(p),f^n(q)}=g_n, p^{E^{su}(p)}_{E^c(p),E^c(q)}=\pi_n$. The base is $\mathbb Z$ with the discrete topology, and the base map is $T$, the translation by one. The fiber over $n$ is
$$
B_n=B(J(g_n), C_1d(f^n(p),f^n(q))^{\theta})\subset J^b(\cW^c_{\delta}(f^n(p)),f^n(p),\cW^c_{\delta}(f^n(q)),f^n(q))
$$
if $n\leq 0$, where $C_1=20C_0$. In particular we have $J(h^u_{f^n(p),f^n(q)})\in B_n$. Observe that since $C_1>C_0$ we have
\begin{equation}\label{eq:welldef}
\mu C_1+\frac{4\mu\|Df\|_{C^{\theta}}}{\|Df\|}<C_1.
\end{equation}

The subset $\mathbb Z^-=\mathbb Z\setminus\mathbb N$ is overflowed by $T$. The bundle map is
$$
F(n,J(h))=(n+1, J(f\circ h\circ f^{-1})).
$$

We claim that $F$ is well defined. For this we have to prove that if $d(J(h),J(g_n))\leq C_1d(f^n(p),f^n(q))$ then $d(J(f\circ h\circ f^{-1}),J(g_{n+1}))\leq C_1d(f^{n+1}(p),f^{n+1}(q))$. Observe that
$$
d(J(f\circ h\circ f^{-1}),J(g_{n+1}))\leq d(J(f\circ h\circ f^{-1}),J(f\circ g_n\circ f^{-1}))+d(J(f\circ g_n\circ f^{-1}),J(g_{n+1})).
$$
On one hand we have
\begin{align*}
d(J(f\circ h\circ f^{-1}),J(f\circ g_n\circ f^{-1}))&\leq Lip (f, f^{n}(q))\cdot d(J(h),J(g_n))\cdot Lip (f^{-1},f^{n+1}(p))\\
&\leq\frac{ \lambda^+_{\Delta^c_{2\epsilon_0}}(f,p,\delta_0)}{\lambda^-_{\Delta^c_{2\epsilon_0}}(f,p,\delta_0)\lambda^-_{\Delta^u_{2\epsilon_0}}(f,p,\delta_0)^{\theta}}C_1d(f^{n+1}(p),f^{n+1}(q))^{\theta}\\
&\leq\mu C_1d(f^{n+1}(p),f^{n+1}(q))^{\theta}.
\end{align*}
On another hand, since $g_n$ and $g_{n+1}$ are differentiable, we have
\begin{align*}
d(J(f\circ g_n\circ f^{-1}),J(g_{n+1}))&=\|D(f\circ g_n\circ f^{-1})-Dg_{n+1}\|\\
&=\|p^{E^{su}_{f^{n+1}(p)}}_{E^c_{f^{n+1}(q)}}\cdot(Df(f^n(q))-Df(f^n(p)))\cdot p^{E^{su}_{f^n(p)}}_{E^c_{f^n(q)}}\cdot Df^{-1}|_{E^c_{f^{n+1}(p)}}\|\\
&\leq\frac{4\mu\|Df\|_{C^{\theta}}}{\lambda^-_{\Delta^c_{2\epsilon_0}}(f,p,\delta_0)\lambda^-_{\Delta^u_{2\epsilon_0}}(f,p,\delta_0)^{\theta}}d(f^{n+1}(p),f^{n+1}(q))^{\theta}\\
&\leq\frac{4\mu\|Df\|_{C^{\theta}}}{\|Df\|}d(f^{n+1}(p),f^{n+1}(q))^{\theta}.
\end{align*}

The estimates above together with the condition \ref{eq:welldef} imply that $F$ is indeed well defined.

Next we modify the distance inside each fiber $B_n$, we let $d_n=\frac d{d(f^n(p),f^n(q))}$. Let $\Sigma^b$ be the space of sections over $\mathbb Z^-$, with the supremum distance $d_{\sup}=\sup_{n\in\mathbb Z^-}d_n$. It is clear that $(\Sigma^b,d_{\sup})$ is a complete metric space.
We claim that $F$ is a uniform bundle contraction over $\mathbb Z^-$.

Let $J(\sigma), J(\sigma')\in B_n$. Then
\begin{align*}
d_{n+1}(J(f\circ\sigma\circ f^{-1}),J(f\circ\sigma'\circ f^{-1}))&=\frac{d_{n+1}(J(f\circ\sigma\circ f^{-1}),J(f\circ\sigma'\circ f^{-1}))}{d(f^{n+1}(p),f^{n+1}(q))}\\
&\leq Lip (f, f^{n}(q))\cdot Lip (f^{-1},f^{n+1}(p))\cdot \frac{d(f^n(p),f^n(q))}{d(f^{n+1}(p),f^{n+1}(q))}\cdot \\
&\ \ \ \ \ \ \ \cdot \frac{d(J(\sigma),J(\sigma'))}{d(f^n(p),f^n(q))}\\
&\leq\frac{ \lambda^+_{\Delta^c_{2\epsilon_0}}(f,p,\delta_0)}{\lambda^-_{\Delta^c_{2\epsilon_0}}(f,p,\delta_0)\lambda^-_{\Delta^u_{2\epsilon_0}}(f,p,\delta_0)^{\theta}}d_n(\sigma,\sigma')\\
&\leq\mu d_n(\sigma,\sigma').
\end{align*}

This shows that $F$ induces a contraction on $\Sigma^b$, so there exists a unique invariant bounded section $\sigma(n)\in B_n$.

$B_n\cap J^d(\cW^c_{\delta}(f^n(p)),f^n(p),\cW^c_{\delta}(f^n(q)),f^n(q))$ is a closed nonempty subset of $B_n$,  so we can apply again the invariant section theorem to this closed sub-bundle, which is clearly preserved by $F$, and we get that the unique invariant section must contain actually differentiable jets at all points.

We can check that the jet of the holonomy is also an invariant bounded section of $F$. Uniqueness of the invariant section implies then that the holonomy is differentiable at every points $p,q\in\cW^u_{\delta}(p)$, and satisfies
$$
\|Dh^u_{p,q}(p)-p^{E^{su}_p}_{E^c_p,E^c_q}\|\leq C_1d(p,q)^{\theta}.
$$

\begin{remark}
We proved the differentiability of the unstable holonomy between (nearby) center leaves. However we can adapt the proof for any two transversals to $\cW^u$ inside a center-unstable leaf. A sketch of the proof is the following.

Let $T_p,T_q$ be two $C^1$ transversals to $\cW^u$ restricted to $\cW^{cu}(p)$, and denote $D_p$ and $D_q$ their tangent planes in $p,q$. Assume that $D_p,D_q\in\Delta^c_{\frac{\epsilon_0}4}$ and $d(p,q)<\delta/4$ (otherwise iterate back a finite number of times). Choose $\tilde\cW^s$ a smooth approximation of $\cW^s$ in a tubular neighborhood of $\cW^{cu}_{\delta}(p)$. The local $\tilde\cW^s$ holonomy takes $T_p,T_q$ to the $C^1$ transversals $T_p^n,T_q^n$ to $f_*^n\cF_n$ inside $f^n\cW_n$. If $n$ is sufficiently large, $f^n\cW_n$ is close to $\cW^{cu}(p)$, and the local $\tilde\cW^s$ holonomy takes $D_p,D_q$ to subspaces $D_p^n,D_q^n$ inside $\Delta^c_{\frac{\epsilon_0}2}$. We do have again the uniform control of the regularity of the $f_*^n\cF_n$ holonomy between $T_p^n$ and $T_q^n$, so we can pass it to the limit as before and obtain that the unstable holonomy between $T_p$ and $T_q$ is differentiable, with
$$
\|Dh^u_{T_p,T_q}(p)-p^{E^{su}_p}_{D_p,D_q}\|\leq C_1d(p,q)^{\theta}.
$$

In other words $Dh^u$ is ($C_1,\theta$)-H\"older along $\cW^u$ at $p$ with respect to $E^{su}_p,\Delta^c_{\frac{\epsilon_0}4}$ and at scale $\delta/4$ for all $p\in M$. The result holds for the neighborhood $\mathcal U(f)$.
\end{remark}

\subsection{$Dh^u_{p,q}$ is continuous in $p,q,f$.}\label{ss:final}

We will apply again the invariant section theorem in yet another space. First let us refine the bunching bound from \ref{eq:bunch}. Choose $\mu<\mu'<1$. Since $E^c$ is uniformly $C^{\alpha}$ in a neighborhood of $f$, there exists $0<\delta'<\delta$ such that for all $g\in\mathcal U(f), p,q\in M, d(p,q)\leq\delta'$, we have 
$$
\|p^{E^{su}_{p,g}}_{E^c_{p,g},E^c_{q,g}}\|, \|p^{E^{su}_{p,g}}_{E^c_{q,g},E^c_{p,g}}\|< \sqrt{\frac{\mu'}{\mu}}.
$$

The base space is $N=M^2\times\mathcal U(f)$, with the $C^1$ topology on $\mathcal U(f)$. The base map is $G(p,q,g)=(g(p),g(q),g)$, which is continuous. At each $(p,q,g)\in N$ we consider the fiber $\mathcal E_{p,q,g}=\mathcal L(E^c_{p,g})$, the linear maps from $E^c_{p,g}$ to itself, with the usual norm given by the Riemannian metric. Since the center bundle is continuous with respect to the point and the map, we obtain a continuous Banach bundle $\mathcal E$ over $N$. Let $N'=\{(p,q,g)\in N:\ q\in\cW^u_{\delta'}(p,g)\}$. Clearly $N'$ is overflowed by $G$.

Let
$$
\|\sigma\|_b=\sup_{(p,q,g)\in N'}\frac{\|\sigma(p,q,g)\|}{d(p,q)^{\theta}}.
$$
and let $\Sigma^b$ the space of sections in $\mathcal E$ over $N'$ bounded in $\|\cdot\|_b$ (in particular $\sigma\in\Sigma^b$ implies $\sigma(p,p,g)=0$). This is a complete metric space. $\Sigma^c\cap\Sigma^b$ is the space of the sections which are both continuous and bounded in $\|\cdot\|_b$, this is a closed nonempty subset of $\Sigma^b$ (it contains the zero section).

The bundle map is
$$
(T\sigma)(g(p),g(q),g)=p^{E^{su}_{g(p),g}}_{E^c_{g(q),g}E^c_{g(p),g}}\circ Dg(q)|_{E^c_{q,g}}\circ p^{E^{su}_{p,g}}_{E^c_{p,g},E^c_{q,g}}\circ(Id+\sigma(p,q,g))\circ Dg(p)|_{E^c_{p,g}}^{-1}-Id.
$$
This is continuous in $p,q,g,\sigma$.

The connection with the holonomies is the following. If
$$
Id+\sigma(p,q,g)=p^{E^{su}_{p,g}}_{E^c_{q,g},E^c_{p,g}}\circ H^u_{p,q,g},
$$
where $H^u_{p,q,g}:E^c_{p,g}\rightarrow E^c_{q,g}$ is the candidate for the derivative of the holonomy, then
$$
Id+(T\sigma)(g(p),g(q),g)=p^{E^{su}_{g(p),g}}_{E^c_{g(q),g},E^c_{g(p),g}}\circ g_*H^u_{p,q,g}=p^{E^{su}_{g(p),g}}_{E^c_{g(q),g},E^c_{g(p),g}}\circ Dg(q)|_{E^c_{q,g}}\circ H^u_{p,q,g}\circ Dg(p)|_{E^c_{p,g}}^{-1}
.$$

Let us check that $T$ applied to the zero section is in $\Sigma^b$. We remark first that
$$
Id=p^{E^{su}_{g(p),g}}_{E^c_{g(q),g}E^c_{g(p),g}}\circ Dg(p)|_{E^c_{q,g}}\circ p^{E^{su}_{p,g}}_{E^c_{p,g},E^c_{q,g}}\circ Dg(p)|_{E^c_{p,g}}^{-1}.
$$
Then
\begin{align*}
\|T0\|_b&=\sup_{(g(p),g(q),g)\in N'}\frac{\|p^{E^{su}_{g(p),g}}_{E^c_{g(q),g},E^c_{g(p),g}}\circ Dg(q)|_{E^c_{q,g}}\circ p^{E^{su}_{p,g}}_{E^c_{p,g},E^c_{q,g}}\circ Dg(p)|_{E^c_{p,g}}^{-1}-Id\|}{d(g(p),g(q))^{\theta}}\\
&\leq\sup_{(p,q,g)\in N'}\frac{\|p^{E^{su}_{g(p),g}}_{E^c_{g(p),g}}\circ (Dg(q)-Dg(p))|_{E^c_{q,g}}\circ p^{E^{su}_{p,g}}_{E^c_{p,g},E^c_{q,g}}\circ Dg(p)|_{E^c_{p,g}}^{-1}\|}{d(g(p),g(q))^{\theta}}\\
&\leq\sup_{(p,q,g)\in N'}\frac{4\|Dg\|_{C^{\theta}}}{\lambda^-_{\Delta^c_{2\epsilon_0}}(g,p,\delta')\lambda^-_{\Delta^u_{2\epsilon_0}}(g,p,\delta')^{\theta}}\\
&\leq\frac{4\mu\|Dg\|_{C^{\theta}}}{\|Dg\|}.
\end{align*}

Now let us check that $T$ is a contraction in $\Sigma^b$.

\begin{align*}
\|T\sigma_1-T\sigma_2\|_b&=\sup_{N'}\frac{\|p^{E^{su}_{g(p),g}}_{E^c_{g(q),g},E^c_{g(p),g}}\circ Dg(q)|_{E^c_{q,g}}\circ p^{E^{su}_{p,g}}_{E^c_{p,g},E^c_{q,g}}\circ(\sigma_1-\sigma_2)(p,q,g))\circ Dg(p)|_{E^c_{p,g}}^{-1}\|}{d(g(p),g(q))^{\theta}}\\
&\leq\frac{\mu' \lambda^+_{\Delta^c_{2\epsilon_0}}(g,p,\delta')}{\mu \lambda^-_{\Delta^c_{2\epsilon_0}}(g,p,\delta')\lambda^-_{\Delta^u_{2\epsilon_0}}(g,p,\delta')^{\theta}}\|\sigma_1-\sigma_2\|_b\\
&\leq\mu'\|\sigma_1-\sigma_2\|_b.
\end{align*}

Since $\Sigma^b$ is a complete metric space, we obtain that there is a unique invariant section in $\Sigma^b$. Continuous sections are preserved by $T$, so we can also apply the Banach fixed point Theorem in $\Sigma^b\cap\Sigma^d$, and we obtain that the unique invariant section in $\Sigma^d$ is in fact continuous. On the other hand the section
$$
\sigma^u(p,q,g)=p^{E^{su}_{p,g}}_{E^c_{q,g},E^c_{p,g}}\circ Dh^u_{p,q,g}-Id
$$
is an invariant section of $T$ inside $\Sigma^b$, so it must be the unique invariant section. Since $p^{E^{su}_{p,g}}_{E^c_{q,g},E^c_{p,g}}$ is continuous in $p,q,g$, we obtain that $Dh^u$ is also continuous in $p,q,g$, which finishes the proof of Theorem \ref{th:t1}.

\begin{remark}
If we consider the restriction to the base space $M^2\times\{f\}$, then we have a H\"older map in a H\"older bundle, so the invariant section theorem will provide us with a H\"older continuous invariant section, which means that $Dh^u_{p,q,f}$ is actually H\"older in $p,q$.
\end{remark}

\subsection{Proof of Corollary \ref{cor}}

The proof is similar to the proof of Theorem \ref{th:t1}. The space is not compact (it is a disjoint union of $\mathbb R^d$), but the bounds are uniform. The invariant foliations are globally defined graphs so in this case the approximation of the pair ($\cW^{cu},\cW^u$) is actually much easier. We can take $\tilde\cW^{cu}$ to be the $cu$-subspace passing through the origin, and the subfoliation $\cF$ to be the subfoliation by $u$-subspaces. For more details on fake foliations we send the reader to \cite{BW10}.

\subsection{Proof of Theorem \ref{th:t2}}

The proof is actually contained in the section \ref{ss:final}. Even if we don't know that there exists a (differentiable) holonomy between center leaves, we still obtain a continuous invariant section $\sigma^u$ of $T$, and then $H^u_{p,q,g}=p^{E^{su}_{p,g}}_{E^c_{p,g},E^c_{q,g}}\circ(\sigma^u_{p,q,g}+Id)$ is the invariant continuous holonomy we are looking for, at least at the scale $\delta'$. In order to define it for all $q\in\cW^u(p)$ we iterate forward and use invariance under $f$. Doing this we have automatically the invariance under $f$ and the continuity with respect to the points. To prove that $H^u_{q,r}\circ H^u_{p,q}=H^u_{p,r}$ we can use the uniqueness of the invariant section. If the relation does not hold, we can modify the invariant section $\sigma$ along the orbit of $(p,r)$, replacing it with the $\sigma'$ corresponding to $H^u_{q,r}\circ H^u_{p,q}$. Then the invariant section $\sigma$ is not unique, which is a contradiction.

\end{document}